\documentclass[11pt]{amsart} 

\usepackage{amsmath}
\usepackage{amsthm}
\usepackage{hyperref} 
\usepackage[hypcap]{caption}
\usepackage[shortlabels]{enumitem}

\newtheorem{mainthm}{Theorem}

\newtheorem{thm}{Theorem}[section]
\newtheorem{cor}[thm]{Corollary}
\newtheorem{lem}[thm]{Lemma}

\theoremstyle{remark}
\newtheorem*{claim}{Claim}

\newcommand{\Bb}{\mathcal B}

\newcommand{\Ss}{\mathcal S}
\newcommand{\Mm}{\mathcal M}

\newcommand{\sd}{\triangle}

\newcommand{\del}{\backslash}
\DeclareMathOperator{\cl}{cl}  
\DeclareMathOperator{\rank}{rank}  
\DeclareMathOperator{\loops}{loops}  
\DeclareMathOperator{\comp}{comp}  
\newcommand{\GB}{(G,\Bb)}
\newcommand{\GBp}{(G',\Bb')}
 
\newcommand{\comout}[1]{}

\begin{document}

\title
[Excluded minors for graphical matroids]
{On excluded minors for classes of graphical matroids}
\date{\today} 
\author[Funk \and Mayhew]{Daryl Funk \and Dillon Mayhew}
\address{Department of Mathematics and Statistics, Victoria University of Wellington, New Zealand}
\email{daryl.funk@vuw.ac.nz} 
\email{dillon.mayhew@vuw.ac.nz}
\thanks{Supported by a Rutherford Discovery Fellowship.}

\begin{abstract}
Frame matroids and lifted-graphic matroids are two distinct minor-closed classes of matroids, each of which generalises the class of graphic matroids.  
The class of {quasi-graphic} matroids, recently introduced by Geelen, Gerards, and Whittle, simultaneously generalises both the classes of frame and lifted-graphic matroids.  
Let $\Mm$ be one of these three classes, and let $r$ be a positive integer.  
We show that $\Mm$ has only a finite number of excluded minors of rank $r$.  
\end{abstract}

\maketitle

A matroid is a \emph{frame matroid} if it may be extended so that it has a basis $B$ such that every element is spanned by at most two elements of $B$.  
Such a basis is a \emph{frame} for the matroid.  
A matroid $M$ is a \emph{lifted-graphic matroid} if there is a matroid $N$ with $E(N) = E(M) \cup \{e\}$ such that $N \del e = M$ and $N/e$ is graphic.  

Frame matroids are a natural generalisation of graphic matroids: 
the cycle matroid $M(G)$ of a graph $G$ is naturally extended by adding its vertex set $V(G)$ as its frame, and declaring each non-loop edge to be minimally spanned by its endpoints. 
Classes of representable frame matroids play an important role in the matroid-minors project of Geelen, Gerards, and Whittle \cite[Theorem
3.1]{GeelenGerardsWhittle:HighlyConnected}, analogous to that of graphs embedded on surfaces in graph structure theory.  

Frame matroids form a minor-closed class.  
Despite its importance, little is known about its excluded minors.  
Zaslavsky has exhibited several in \cite{MR1273951}.  
Bicircular matroids are a relatively well-studied proper minor-closed class of frame matroids; it is known that this class has only a finite number of excluded minors \cite{exminorsforbicircular}.  
There are also natural proper minor-closed classes of frame matroids, and of lifted-graphic matroids, that have, for any fixed $r \geq 3$, infinitely many excluded minors of rank $r$ \cite{MR3267062}.  
The first systematic study of excluded minors for the class of frame matroids is \cite{MR3588716}, in which 18 excluded minors of connectivity 2 for the class of frame matroids is exhibited, and it is proved that any other excluded minor of connectivity 2 is a 2-sum of a 3-connected non-binary frame matroid with $U_{2,4}$.  
The class of lifted-graphic matroids is minor-closed.  
Less is known of their excluded minors than those for frame matroids.  

Here we prove the following theorems.  

\begin{mainthm} \label{thm:finitenumberexminfixedrankFRAME} 
Let $r$ be a positive integer.  
There are only a finite number of excluded minors of rank $r$ for the class of frame matroids.  
\end{mainthm} 

\begin{mainthm} \label{thm:finitenumberexminfixedrankLIFT} 
Let $r$ be a positive integer.  
There are only a finite number of excluded minors of rank $r$ for the class of lifted-graphic matroids.  
\end{mainthm} 

The prevailing belief among members of the matroid community was that each of these classes should have only a finite number of excluded minors.  
However, Chen and Geelen \cite{Chen_Geelen_frame} have recently settled the question, rather surprisingly, by exhibiting, for each class, an infinite family of excluded minors. 
Each family consists of a sequence of matroids $(M_r)_{r \geq 7}$, defined using a sequence of graphs $(G_r)_{r \geq 7}$, with each excluded minor $M_r$ having rank $r$.  
Theorems \ref{thm:finitenumberexminfixedrankFRAME} and \ref{thm:finitenumberexminfixedrankLIFT} say that, like Chen and Geelen's families, every infinite collection of excluded minors for these classes must contain matroids of arbitrarily large rank.  
Chen and Geelen point out that, ``The existence of an infinite set of excluded minors does not necessarily prevent us from describing a class explicitly"; they point to Bonin's excluded minor characterisation of lattice-path matroids \cite{MR2718679} as an encouraging example.  
Chen and Geelen's two infinite families of excluded minors are of a similar flavour to Bonin's infinite collections of excluded minors: Bonin's characterisation has three easily described infinite sequences of excluded minors, each consisting of a set of matroids indexed by the positive integers, of ever increasing and unbounded ranks.  
Theorems \ref{thm:finitenumberexminfixedrankFRAME} and \ref{thm:finitenumberexminfixedrankLIFT}, therefore, may be seen as support for Chen and Geelen's optimism. 

In \cite{Quasi}, Geelen, Gerards, and Whittle introduce the class of {quasi-graphic matroids}, as a common generalisation of the classes of frame and lifted-graphic matroids.  
For a vertex $v$, denote by $\loops(v)$ the set of loops incident to $v$.  
Given a matroid $M$, a \emph{framework} for $M$ is a graph $G$ satisfying
\begin{enumerate} 
\item $E(G) = E(M)$, 
\item for each component $H$ of $G$, $r(E(H)) \leq |V(H)|$, 
\item for each vertex $v \in V(G)$, $\cl(E(G-v)) \subseteq E(G-v) \cup \loops(v)$, and 
\item if $C$ is a circuit of $M$, then the graph induced by $E(C)$ has at most two components.  
\end{enumerate} 
A matroid is \emph{quasi-graphic} if it has a framework.  
Chen and Geelen conjecture that the class of quasi-graphic matroids has only finitely many excluded minors \cite{Chen_Geelen_frame}.  
We prove that, like the classes of frame and of lifted-graphic matroids, when fixing the rank this is indeed the case. 

\begin{mainthm} \label{thm:finitenumberexminfixedrankQUASI} 
Let $r$ be a positive integer.  
There are only a finite number of excluded minors of rank $r$ for the class of quasi-graphic matroids.  
\end{mainthm} 

\section{Preliminaries} 
\label{sec:Preliminaries} 

\subsection{Frame matroids} 
Zaslavsky \cite{MR1273951} has shown that the class of frame matroids is precisely that of matroids arising from \emph{biased graphs}, as follows.  
Let $M$ be a frame matroid on ground set $E$, with frame $B$.  
By adding elements in parallel if necessary, we may assume $B \cap E = \emptyset$.  
Hence for some matroid $N$, $M = N \del B$ where $B$ is a basis for $N$ and every element $e \in E$ is minimally spanned by either a single element or a pair of elements in $B$.  
Let $G$ be the graph with vertex set $B$ and edge set $E$, in which $e$ is a loop with endpoint $f$ if $e$ is parallel with $f \in B$, and otherwise $e$ is an edge with endpoints $f, f' \in B$ if $e \in \cl\{f,f'\}$.  
The edge set of a cycle of $G$ is either independent or a circuit in $M$.  
A cycle $C$ in $G$ whose edge set is a circuit of $M$ is said to be \emph{balanced}; otherwise $C$ is \emph{unbalanced}.  
Let $\Bb$ be the collection of balanced cycles of $G$.  
The pair $(G,\Bb)$ is called a \emph{biased graph}; one may think of the pair as a graph equipped with the extra information of the \emph{bias}---balanced or unbalanced---of each of its cycles.  
A \emph{theta graph} consists of a pair of distinct vertices and three internally disjoint paths between them.  
The circuits of $M$ are precisely those sets of edges inducing one of: 
a balanced cycle, 
a theta subgraph in which all three cycles are unbalanced, 
two edge-disjoint unbalanced cycles intersecting in just one vertex, or 
two vertex-disjoint unbalanced cycles along with a minimal path connecting them.  
The later two biased subgraphs are called \emph{handcuffs}, \emph{tight} or \emph{loose}, respectively.  
Such a biased graph $\GB$ \emph{represents} the frame matroid $M$, and we write $M = F\GB$.  
Since the collection $\Bb$ of balanced cycles of $G$ is determined by the matroid $M$, we may speak simply of the graph $G$ as a frame representation of $M$, with its collection of balanced cycles being understood as implicitly given by $M$. 
Thus we may unambiguously refer to $G$ as a \emph{frame graph for $M$}, or say that $M$ \emph{has} a frame graph $G$. 
When it is clear from context that $G$ is a frame representation of $M$, we say simply that $G$ is a \emph{graph} for $M$ or that $M$ \emph{has} a graph $G$. 

\subsection{Lifted-graphic matroids} 
Let $N$ be a matroid on ground set $E \cup \{e\}$, and suppose $G$ is a graph with edge set $E$ and with cycle matroid $M(G)$ equal to $N/e$.  
Then $M = N \del e$ is a lifted-graphic matroid.  
Each cycle in $G$ is either a circuit of $N$, or together with $e$ forms a circuit of $N$.  
Again, cycles whose edge set is a circuit of $M$ are said to be balanced, and those whose edges form an independent set are unbalanced.  
Zaslavsky has shown \cite{Zaslavsky:BG2} that the circuits of $M$ are precisely those sets of edges inducing one of: a balanced cycle, a theta subgraph in which all three cycles are unbalanced, two edge disjoint unbalanced cycles meeting in just one vertex, or two vertex-disjoint unbalanced cycles.  
Letting $\Bb$ denote the collection of balanced cycles of $G$, we again say the biased graph $\GB$ so obtained \emph{represents} the lifted-graphic matroid $M$; we write $M = L\GB$.  
As with frame matroids, we may more simply say $G$ is a \emph{lift graph for $M$}, or that $M$ \emph{has} a lift graph $G$, with its collection of balanced cycles being implicitly given by $M$.  
Similarly, when clear in context that $G$ is a lift graph for $M$, we may say simply that $G$ is a \emph{graph} for $M$, and that $M$ \emph{has} the graph $G$. 

\subsection{Quasi-graphic matroids} 
In \cite{Quasi}, the authors show that the class of quasi-graphic matroids includes the classes of frame and lifted-graphic matroids, by showing that if $G$ is a graph for a frame matroid $M$, then $G$ is a framework for $M$, and that similarly, if $G$ is a graph for a lift matroid $N$, then $G$ is a framework for $N$.  
For the sake of consistency and to aid exposition, if $M$ is quasi-graphic and $G$ is a framework for $M$, we as well refer to $G$ as a graph \emph{for} $M$, or say that $M$ \emph{has} a graph $G$.  

Note that in general, given a frame matroid $M$ there may be many frames for $M$, and so many frame graphs for $M$.  
Similarly, given a lift $N$ of a graphic matroid, there may be many graphs with cycle matroid $N/e$, and so many lift graphs for $N$.  
Of course, nor is there any guarantee that a quasi-graphic matroid has a unique framework. 

Let $M$ be a quasi-graphic matroid, and $G$ a graph for $M$.  
Again, call these cycles of $G$ that are circuits of $M$ \emph{balanced}.  
In \cite[Lemma 3.3]{Quasi} it is shown that the circuits of $M$ are precisely those edge sets appearing in $G$ as a balanced cycle, or one of the following subgraphs: 
a theta in which all cycles are unbalanced, 
a pair of unbalanced cycles meeting in exactly one vertex (tight handcuffs), 
a pair of vertex disjoint unbalanced cycles together with a minimal path linking them (loose handcuffs), or 
a pair of vertex disjoint unbalanced cycles. 

In any case, whether $M$ is frame, lifted-graphic, or more generally, quasi-graphic, if $G$ is a graph for $M$, then the collection $\Bb$ satisfies the following property: 
if $C$ and $C'$ are two balanced cycles contained in a common theta subgraph $T$ of $G$, then the third cycle $C \sd C'$ contained in $T$ is also balanced.   
Any collection of cycles with this property is said to satisfy the \emph{theta property}. 
Zaslavsky has shown \cite{Zaslavsky:BG2} that conversely, in order that an arbitrary collection $\Bb$ of cycles of a graph $G$ define a frame matroid $F\GB$ or lifted-graphic matroid $L\GB$, it is only necessary that $\Bb$ be chosen so that it satisfies the theta property.  

\subsection{Working with graphs for matroids} 
If $G$ is a graph for $M$, and $e \in E(M)=E(G)$, we refer to $e$ as a point, an element, or an edge.  
It is at times convenient that paths, cycles, and induced subgraphs include their vertices, and at other times more convenient that they consist of just their edge sets.  
We will be explicit should context fail to make clear which object we have in mind.  
Edges may be referred to as \emph{loops}, having just one endpoint, or \emph{links}, having two distinct endpoints, when it is important that the distinction be clear. 
We write $e=uv$ to indicate that the link $e$ has distinct endpoints $u$ and $v$.  

For a subset $X \subseteq E(M)$ and a graph $G$ for $M$, we write $G[X]$ for the subgraph of $G$ induced by $X$.  
Evidently, $G[X]$ is a graph for $M|X$.  
Likewise, for a subset $U \subseteq V(G)$ of vertices, we write $G[U]$ for the subgraph of $G$ induced by $U$.  
If $X$ is a set of edges, we denote by $V(X)$ the set of vertices of $G$ incident to an edge in $X$.  
If $G[X]$ contains no unbalanced cycle, then $G[X]$ is \emph{balanced}; otherwise $G[X]$ is \emph{unbalanced}.  
If $M$ is a frame matroid, then the rank of a set $X \subseteq E(M)$ is given by $r(X) = |V(X)| - b(X)$, where $b(X)$ is the number of balanced components of $G[X]$.  
Thus if $M$ is connected, $r(M)=|V(G)|$, unless $G$ is balanced.  
Since a balanced biased graph represents a graphic matroid, and the frame matroids we consider here are connected and non-graphic, we will always have $r(M)=|V(G)|$ when $G$ is a graph for a frame matroid $M$.  
If $M$ is a lifted-graphic matroid, then the rank of a set $X \subseteq E(M)$ is given by $r(X) = |V(X)| - c(X) + \delta$, where $c(X)$ is the number of components of $G[X]$, and $\delta = 0$ if $G[X]$ is balanced and $\delta=1$ otherwise.  
Thus if $M$ is a non-graphic lifted-graphic matroid, and $G$ is a connected graph for $M$, we have $r(M) = |V(G)|$.  

\subsection{Excluded minors are simple} 
The first sentence in the following lemma appears in \cite{MR3588716}. 
The second sentence has a similarly straightforward proof as the first.  

\begin{lem} \label{lem:ex_min_simple_cosimple} 
An excluded minor for the class of frame matroids is connected, simple, and cosimple.  
An excluded minor for the class of lifted-graphic matroids is simple and cosimple.  
\end{lem}

Note that the class of lifted-graphic matroids is not closed under direct sum: $U_{2,4}$ is lifted-graphic (the graph consisting of four edges linking a pair of vertices is a graph for $U_{2,4}$), but the direct sum of two copies of $U_{2,4}$ is an excluded minor for the class.  
The class of quasi-graphic matroids, like that of frame matroids, is closed under direct sum: if $G_1$ is a graph for $M_1$, and $G_2$ is a graph for $M_2$, then evidently the disjoint union of $G_1$ and $G_2$ is a graph for the direct sum of $M_1$ and $M_2$.  
We include a proof of the following elementary lemma, since the class has only recently been introduced.  

\begin{lem} \label{lem:ex_min_simple_cosimple_QUASI} 
An excluded minor for the class of quasi-graphic matroids is connected, simple and cosimple.  
\end{lem}

\begin{proof} 
Let $M$ be an excluded minor for the class of quasi-graphic matroids.  
The fact that the class is closed under direct sum implies that $M$ is connected.  
Suppose $M$ has a loop $e$.  
There is a graph $G$ for $M \del e$.   
Adding a loop for $e$ incident to any vertex of $G$ yields a graph for $M$, a contradiction.  
Similarly, if $M$ has a coloop $f$, consider graph $G$ for $M/f$.  
Adding a new vertex $w$, choosing any vertex $v \in V(G)$, and adding edge $f= vw$ to $G$ yields a graph for $M$, a contradiction.  

Now suppose $M$ has a 2-element circuit $\{e,f\}$.  
Let $G$ be a graph for $M \del e$.  
If $f$ is a link in $G$, say $f = uv$, then let $G'$ be the graph obtained from $G$ by adding $e$ in parallel with $f$ so $e$ also has endpoints $u$ and $v$.  
If $f$ is a loop in $G$, say incident to $u \in V(G)$, then let $G'$ be the graph obtained from $G$ by adding $e$ as a loop also incident with $u$. 
Now $G'$ is a graph for $M$, a contradiction.  
Similarly, if $\{e,f\}$ is a series class in $M$, let $G$ be a graph for $M/e$.  
If $f$ is a link in $G$, say $f=uv$, then let $G'$ be the graph obtained from $G$ by deleting $f$, adding a new vertex $w$, and putting $f=uw$ and $e=wv$.  
If $f$ is a loop in $G$, say incident to $u \in V(G)$, let $G'$ be the graph obtained from $G$ by deleting $f$, adding a new vertex $w$, and adding edges $e$ and $f$ in parallel, both with endpoints $u,w$.   
Again, $G'$ is a graph for $M$, a contradiction.  
\end{proof} 

Because of Lemmas \ref{lem:ex_min_simple_cosimple} and \ref{lem:ex_min_simple_cosimple_QUASI}, we may omit the adjective ``unbalanced'' when speaking of loops---all loops in biased graphs in this paper are unbalanced. 

\subsection{Long lines, fixing sets, fixing graphs}
A \emph{line} is a rank 2 flat.  
A \emph{long line} is a line containing at least 6 rank-1 flats.  
If $l$ is a long line in a simple frame matroid $M$, and $G$ is a graph for $M$, then there are just three non-isomorphic possibilities for $G[l]$: 
each consists of a pair of vertices with $|l|-i$ links between them and $i$ loops, where $i \in \{0,1,2\}$, respectively, and if $i=2$ then the two loops are incident to different vertices; each has all cycles unbalanced.  
If $l$ is a long line in a simple lifted-graphic matroid $M$ and $G$ is a graph for $M$, then there are just three non-isomorphic possibilities for $l$: 
a pair of vertices with $|l|$ links between them, 
a pair of vertices $u,v$ with $|l|-1$ links between them and a single loop incident to $u$ or $v$, or 
a pair of vertices $u,v$ with $|l|-1$ links between them and a single loop incident to a vertex $w \notin \{u,v\}$.  
Hence if $l$ is a long line in a simple frame matroid, then in any graph for the matroid, there is a pair of vertices $u,v$ such that $l = E(G[\{u,v\}])$.  
If $l$ is a long line in a simple lifted-graphic matroid, then in any graph for the matroid, either there is a pair of vertices $u,v$ such that $l = E(G[\{u,v\}])$ or there are three vertices $u,v,w$ such that $l = E(G[\{u,v\}]) \cup \{e\}$, where $e$ is a loop incident to $w$.  
Furthermore, observe that in the latter case, $e$ is contained in every long line of $M$, and $\{e\}$ is the only loop in $G$. 

In fact, if $l$ is a long line in any simple quasi-graphic matroid $M$, and $G$ is a graph for $M$, then the non-isomorphic possibilities for $G[l]$ are just those described above for the cases that $M$ is frame or lifted-graphic. 
Thus, in any case, if $l$ is a long line in a matroid $M$ and $G$ is a graph for $M$, then there is a pair of vertices $u,v$ that we may unambiguously associate with $l$, namely, that pair which all elements of $l$ that are links in $G$ share as their endpoints.  
This pair of vertices are the \emph{endpoints} of $l$ in $G$.  

The \emph{fixing set} of a matroid $M$ is the set $X(M) = \{e \in E(M) : e \in l$ and $l$ is a long line of $M\}$.  
Two long lines are \emph{adjacent} if their union has rank 3. 
A \emph{component} $X'$ of the fixing set $X(M)$ of a matroid $M$ is a set $\bigcup_{l \in L} \{e : e \in l\}$ of elements contained in a nonempty maximal set of long lines $L$ such that if $l, l' \in L$ then $L$ contains a sequence of long lines $l_0, l_1, \ldots, l_t$ with $l=l_0$ and $l_t=l'$ such that each consecutive pair is adjacent.  
Let $M$ be a simple rank-$r$ frame or lifted-graphic matroid, with nonempty fixing set $X=X(M)$.  
Let $G$ be a graph for $M$.  
Let $W=W(G)$ be the graph obtained by deleting all edges of $G$ not contained in a long line, deleting loops, and replacing each parallel class of edges remaining with a single edge. 
Equivalently, $W$ is the simple graph on $V(G)$ in which vertices $u$ and $v$ are adjacent if they are the pair of endpoints for a long line in $G$.  
We call $W$ the \emph{fixing graph} of $G$, and associate to each edge $e_l = uv$ of $W$ the unique line $l$ of $M$ with endpoints $u$ and $v$.  

\subsection{Overview} 
These preliminaries in place, let us now briefly discuss our overall strategy. 
Let $\Mm$ be one of the classes of frame, lifted-graphic, or quasi-graphic matroids, and let $M$ be a rank-$r$ excluded minor for $\Mm$.  
We prove each of Theorems \ref{thm:finitenumberexminfixedrankFRAME}, \ref{thm:finitenumberexminfixedrankLIFT}, and \ref{thm:finitenumberexminfixedrankQUASI} by showing that $M$ does not contain an arbitrarily long line as a restriction. 
This bounds $|E(M)|$. 
We do this by choosing as \emph{canonical} those graphs for matroids in $\Mm$ with certain properties, then establishing a bound on the number of canonical graph representations for a matroid in $\Mm$.  
If $M$ has a sufficiently long line $l$, then we can use this bound to find three elements $e, f, g \in l$ such that there is a graph $G_e$ for $M \del e$, a graph $G_f$ for $M \del f$, and a graph $G_g$ for $M \del g$, such that $G_e \del f,g = G_f \del e,g = G_g \del e,f$. 
But these graphs can be used to construct a graph for $M$, a contradiction.  

The strategy for bounding the number of canonical graph representations for a frame, lifted-graphic, or quasi-graphic matroid is essentially the same in each case.  
The key is bounding the number of ways in which the elements in long lines may appear in a canonical graph representation for a matroid in each of these classes. 
That is, for a matroid $M \in \Mm$ with fixing set $X$, though there may be arbitrarily many graphs for $M$, we bound the number of induced subgraphs $G[X]$ that may appear in a canonical graph $G$ for $M$. 
So, though the fixing set $X$ of $M$ may contain arbitrarily many elements, there is a bound on the number of subgraphs that $X$ induces among all graphs for $M$.  

Our bounds are na\"ive, though somewhat less so for frame matroids.  
We bound the number of ways the fixing set may appear in a canonical graph representation by bounding the number of possible fixing graphs for a matroid. 
If $M$ is a frame matroid, this bound is sharp: there is just one (Lemma \ref{lem:fixing_graphs_unique}). 
Lifted-graphic and quasi-graphic matroids may in general have many fixing graphs; our bounds are therefore cruder for these classes.  
In any case, we are able to show that for each fixing graph, the remainder of a canonical graph for $M$ is determined by the graph representation on $E(M)-X$. 
There may, in general, be many graphs for $M$ that agree on their fixing graph but differ on $E(M)-X$.  
We completely side-step the problem of determining these graphs, as follows. 
For any matroid $M \in \Mm$, the number of elements outside the fixing set of $M$ is bounded. 
We simply bound the number of canonical graphs for $M$ by the number of graphs on $r$ vertices and $|E(M)-X|$ edges. 

The fact that a frame matroid $M$ has a unique fixing graph enables us show that a canonical frame graph representation for the restriction of $M$ to the closure of its fixing set is unique (Theorem \ref{lem:unique_rep_in_span_of_fixing_set}).  
The proof of this fact is the longest of the paper.  
We could, in fact, do without it, at the cost of a somewhat cruder bound on the number of canonical graphs for a simple frame matroid (for a matroid of rank $r$, worse than the bound of Lemma \ref{lem:boundoncanonicalrepsFRAME} by a factor of $r^{r^2}$).  
We believe the effort worthwhile, as the result may be of independent interest.

\section{Frame matroids} \label{sec:frame} 

Throughout this section, for a frame matroid $M$ and a graph $G$ for $M$, we always mean that $G$ is a frame graph for $M$. 
Let $M$ be a simple connected frame matroid.  
Observe that two long lines in a frame matroid are adjacent if and only if they share a vertex in any graph representation.  
Observe also that an element $e$ is contained in more than one long line if and only if these lines are adjacent and $e$ is represented as a loop incident to their shared endpoint in any graph representation.  
An element contained in just one long line may also be represented as a loop, in certain circumstances.  
We say $l$ is a \emph{pendant line} of $M$ if $l$ is a line with at least three elements, and $(l,E-l)$ is a 2-separation of $M$.  
Let $G$ be a graph for $M$, with nonempty fixing set $X$. 
Then $e_l$ is a pendant edge in a fixing graph of $G$ if and only if $e_l$ corresponds to a pendant line $l$ of $M|X$.  
If $u,v$ are the endpoints of $l$ in $G$, and $V(l) \cap V(E-l) = \{u\}$, then we call $v$ the \emph{pendant vertex} of $l$. 

We now describe a local modification that changes the number of loops in a representation having a pendant line. 
Let $l$ be a pendant line in a simple frame matroid $M$, let $G$ be a graph for $M$, and let $u,v$ be the endpoints of $l$ in $G$, where $v$ is the pendant vertex of $l$. 
It is easy to check that if $e \in l$ is a loop incident to $v$, then replacing this loop with a $u$-$v$ edge representing $e$ yields another graph for $M$.  
Conversely, if $G[l]$ has no loop incident to $v$, and $e$ is a $u$-$v$ link, then replacing this link with a loop incident to $v$ representing $e$ yields another graph for $M$.  
We refer to this second operation as a \emph{pendant roll-up}.  
Observe that if $l$ contains $k$ elements, then by successively replacing loops with links while applying pendant roll-ups, we obtain $k+1$ different graphs for $M$. 
Thus a frame matroid with a pendant line may have arbitrarily many biased graph representations. 
The case of a pendant line is quite tame; there are also highly connected frame matroids with arbitrarily many biased graph representations, even when fixing rank 
(see \cite[Section 1.3]{MR3588716} for details).  

To deal with this issue, among all frame graphs for $M$, we choose as \emph{canonical} those that have the least number of loops.  
The following lemma is the key to our proof of Theorem \ref{thm:finitenumberexminfixedrankFRAME}.  
If $G$ is a graph for a simple frame matroid $M$, and $W$ consists of a spanning tree of each component of a fixing graph of $G$, put $Z(W) = \{ e : e \in l$ and $e_l \in W\}$.  
Call $G[Z(W)]$ a \emph{fixing tree} of $G$.  

\begin{thm} \label{lem:unique_rep_in_span_of_fixing_set}
Let $M$ be a simple connected frame matroid, and let $X \subseteq E(M)$ be a component of the fixing set of $M$.  
Then $M|\cl(X)$ has a unique canonical graph $J$.  
Moreover, if $G$ is a graph for $M$, then $G[\cl(X)]$ is obtained from $J$ by pendant roll-ups in a fixing tree of $J$.  
\end{thm} 

We use the following result of Whitney. 
The \emph{line graph} of a graph $G$ is the graph whose vertices are the edges of $G$, in which two vertices $e, e'$ are adjacent if $e$ and $e'$ are incident to a common vertex in $G$. 

\begin{thm}[\cite{MR0256911}, Theorem 8.2] \label{thm:WhitneyLineGraphs}
Let $G$ and $G'$ be connected graphs with isomorphic line graphs.  
Then $G$ and $G'$ are isomorphic unless one is $K_3$ and the other is $K_{1,3}$.  
\end{thm}

\begin{lem} \label{lem:fixing_graphs_unique} 
Let $M$ be a frame matroid.  
Then $M$ has a unique fixing graph.  
\end{lem} 

\begin{proof} 
If $M$ has no long line, then every graph for $M$ has empty fixing graph, so the lemma holds. 
Otherwise, the fixing graph of a graph for $M$ is the disjoint union of the fixing graphs of the components induced by the components of the fixing set of $M$.  
Hence we may assume that $M$ is spanned by its fixing set $X$, and that $X$ has just one component. 
Let $G$ and $G'$ be two graphs for $M$.  
Since $M|X$ is connected, $G[X]$ and $G'[X]$ are connected subgraphs of $G$ and $G'$, respectively, and so their respective fixing graphs are connected.   
The adjacencies of edges in the fixing graphs of $G$ and $G'$ are determined by the adjacencies of the long lines of $M$.  
Since $\rank(M) = |V(G)| = |V(G')|$, it cannot occur that the fixing graph of one $G$ or $G'$ is $K_3$ while the other is $K_{1,3}$.  
Hence by Theorem \ref{thm:WhitneyLineGraphs}, $G$ and $G'$ have the same fixing graph.  
\end{proof} 

Hence we may unambiguously speak of the fixing graph of a frame matroid.  

We prove Lemma \ref{lem:unique_rep_in_span_of_fixing_set} by choosing a special spanning subgraph $S$ of a canonical graph for $M|\cl(X)$, where $S$ contains many bases of $M|\cl(X)$.  
We then construct a biased graph $(H,\Ss)$ by considering fundamental circuits of elements with respect to these bases.  
This biased graph does not necessarily represent $M|\cl(X)$, but is not too far off.  
Local modifications, in the form of pendant roll-ups, are all that may be required to fix this.  
More precisely, if $l$ is a pendant line of $M|\cl(X)$, and $e' \notin \cl(E-l)$, then, while $e'$ may be represented as a loop or a link in a biased graph representation of $M|\cl(X)$, $e'$ is always represented by a link in a canonical representation; in this case our procedure for constructing $(H,\Ss)$ from $S$ places $e'$ as a link, as required.  
If $e'$ is otherwise an element that may be represented as either a loop or a link with respect to a spanning tree of the fixing graph, then our procedure always represents $e'$ as a link in $H$, which may require repair.  

\begin{proof}[Proof of Lemma \ref{lem:unique_rep_in_span_of_fixing_set}]
Write $N = M|\cl(X)$.  
Let $\GB$ and $\GBp$ be two canonical representations of $N$.  
If $N$ has rank 2, then there is just one canonical representation of $N$, consisting of a pair of vertices and $|E(N)|$ links between them, with all cycles unbalanced.  
So assume $\rank(N)>2$.  
By the definition of a component of fixing set, $M|X$, and so $N$, is connected.  
Hence the fixing graph of $N$ is connected.  
By Lemma \ref{lem:fixing_graphs_unique}, the fixing graphs of $G$ and $G'$ are equal.  

Each long line has at least six points, so for each edge $e_l =uv$ of the fixing graph, there are at least two elements of its corresponding line $l$ that are links with endpoints $u, v$ in both $G$ and $G'$.  
Choose a spanning tree $W$ of the fixing graph of $N$.  
Let $S$ be the biased subgraph of $\GB$ and $\GBp$ obtained by replacing each edge $e_l$ of $W$ with two elements of its corresponding long line $l$ that are links in both $G$ and $G'$.  
Evidently $S$ is a biased subgraph of both $\GB$ and $\GBp$.  
Let $T$ be a spanning tree of $S$.  
For each edge $e \in S - T$, put $T_e = T \cup e$.  
For each such edge $e$, $T_e$ is a basis of $N$.  

Now for each element $e' \in E(N) - S$, define  
\[
P_{e'} = \bigcap_{e \in S-T} C(T_e, e') - e'
\]
where $C(T_e, e')$ denotes the fundamental circuit in $N$ containing $e'$ with respect to the basis $T_e$.  
Each fundamental circuit $C(T_e, e')$ defines a subgraph of $\GB$.  
Thus for each element $e' \in E(N)-S$, $P_{e'}$ is a subgraph of $S \subseteq \GB$.  

\begin{claim} 
$P_{e'}$ is either a single vertex, a single edge of $T$, or the path in $T$ linking the distinct endpoints of $e'$ in $G$.  
\end{claim} 

\begin{proof}[Proof of Claim] 
\renewcommand\qedsymbol{$\triangle$}
Fix $e' \in E(N)-S$. 
Since $\rank(N)>2$, $|V(S)|>2$ and $|S-T|>1$. 
Hence there are at least two fundamental circuits appearing in the intersection defining $P_{e'}$, and so $P_{e'} \subseteq T$.  
Since for each $e \in S-T$, $T_e$ is a spanning tree along with one additional edge in an unbalanced 2-cycle, each fundamental circuit $C(T_e, e')$ has in $\GB$ the form of one of: (1) a balanced cycle, (2) a theta subgraph, or (3) a pair of handcuffs, where in the latter two cases all cycles are unbalanced and one cycle is the unbalanced 2-cycle of $T_e$.  
Hence for each $e \in S-T$, $C(T_e, e')-\{e,e'\}$ is in cases (1) and (2) a path contained in $T$ linking the endpoints $x, y$ of $e'$, while in case (3), it is a path $P$ in $T$ linking $x$ and $y$, together with the path in $T$ linking $P$ and the unbalanced 2-cycle of $T_e$.  

Put $Z=Z(W) = \{ e : e \in l$ and $e_l \in W\}$.  
Consider the form of $C(T_e,e')$ in $\GB$.  
If $e' \in Z$, then for all $e \in S-T$ with $|C(T_e,e')| \geq 4$, $C(T_e,e')$ is a pair of handcuffs in $\GB$; if $|C(T_e,e')|=3$, then $e$ is contained in the same line as $e'$  and $C(T_e,e')$ is either a tight handcuff in which $e'$ is a loop or a theta consisting of three edges.  
If $e' \notin Z$, then in $\GB$ edge $e'$ has distinct endpoints $x, y$ that are not adjacent in $S$ (because $M$ is simple and long lines are flats, any edge of $\GB$ that is a loop or shares endpoints with an edge in $S$ is in $Z$).  
Thus if $e' \notin Z$, then for every $e \in S-T$, $C(T_e,e')$ contains the $x$-$y$ path $P \subseteq T$; $P$ has length at least two, and for each edge $s \in S-T$ that is in parallel to an edge in $P$, $C(T_s,e')$ is either a theta with at least 4 edges or a balanced cycle.  
Therefore $e' \in Z$ if and only if every fundamental circuit $C(T_e,e')$ is a handcuff or theta consisting of three edges in $\GB$.  

Moreover, if $e' \in Z$, then for each long line $l$ of $W$ containing $e'$ there are edges $t_l \in T$ and $s_l \in S-T$ such that $\{e',s_l,t_l\}$ is a fundamental circuit $C(T_{s_l},e')$ of size 3.  
Thus if $e' \in Z$, $P_{e'}$ is either a single edge (if $e'$ is contained in just one long line of $W$) or trivial (if $e'$ is contained in more than one long line of $W$).  

Assume now $e' \notin Z$.  
As above, let $x, y$ be the distinct endpoints of $e'$ in $G$, and let $P$ be the $x$-$y$ path in $T$.  
Since $P$ has length at least two, there are at least two edges $e \in S-T$ such that $C(T_e,e')$ is either a theta with at least 4 edges or a balanced cycle.  
Since in the intersection $\bigcap_{e \in S-T} C(T_e,e')$ defining $P_{e'}$, $P$ is contained in every fundamental circuit, and there are at least two fundamental circuits in the intersection, $P_{e'} =P$.  
\end{proof} 

Now define a graph $H$ by adding each element $e' \in E(N)-S$ to $S$ as follows.  
\begin{itemize} 
\item[(1)]  If $P_{e'}$ consists of a single vertex $v$, then $e'$ is a loop incident to $v$. 
\item[(2)]  If $P_{e'}$ is a $u$-$v$ path in $T$, where $u \not= v$, then $e'$ is a $u$-$v$ edge. 
\end{itemize} 
Let $\Ss$ be the collection of cycles of $H$ that are circuits of $N$.  

Put $Z = Z(W) = \{ e : e \in l$ and $e_l \in W\}$.  

\begin{claim} 
$H$ agrees with $G$ up to pendant roll-ups applied in the fixing tree $H[Z]$.  
\end{claim} 

\begin{proof}[Proof of claim] 
\renewcommand\qedsymbol{$\triangle$}
Clearly $H$ agrees with $G$ on $S$.  
Consider an element $e' \in Z - S$.  
\begin{itemize} 
\item If $e'$ is contained in more than one long line of $N|Z$, then in $G$, $e'$ must be a loop incident to a non-pendant vertex $v$ at which all long lines containing $e'$ meet; in this case $P_{e'}$ consists of the single vertex $v$, and $e'$ is added to $S$ according to (1) above: $G$ and $H$ agree on $e'$.  
\item If $e'$ is contained in just one long line $l$ of $N|Z$, and $l$ is not a pendant line of $N|Z$, then in $G$, $e'$ must be an edge linking the endpoints of $l$ 
($l$ shares each of its endpoints with another line of $N|Z$; were $e'$ a loop, $e'$ would be contained in one of these lines); in this case $P_{e'}$ consists of a single link in $T$ between the endpoints of $l$, and $e'$ is added to $S$ according to (2) above: $G$ and $H$ agree on $e'$.  
\item If $e'$ is contained in just one long line $l$ of $N|Z$, and $l$ is a pendant line of $N|Z$, then in $G$, $e'$ may be either a edge between the endpoints of $l$ or a loop incident to the pendant vertex of $l$ in $G[Z]$; in this case $P_{e'}$ consists of a single link in $T$ between the endpoints of $l$, and $e'$ is added to $S$ according to (2) above: $H$ and $G$ agree on $e'$ up to a pendant roll-up in $H[Z]$.  
\end{itemize} 

Finally, consider elements $e' \notin Z$.  
If $P_{e'}$ consists of a single vertex $v$, this is because there is a pair of fundamental circuits $C, C'$ meeting only in $e'$.  
Because in $G$ each of these circuits contains an edge in $T$ incident to $v$, each of which is contained in a long line, this implies that $e'$ is a loop incident to $v$ in $G$.  
Let $l$ be the long line in $Z$ that has $v$ as an endpoint and contains an edge in $T \cap C$.  
Then $e' \in l \subseteq Z$, a contradiction.  
If $P_{e'}$ consists of a single edge, this is because there is a fundamental circuit $C(T_s,e')$ consisting of just three elements $\{s,t,e'\}$, where $s \in S-T$ and $t \in T$.  
But there is a long line $l \in W$ containing both $s$ and $t$, so this again yields the contradiction $e' \in l \subseteq Z$.  
Hence for all remaining elements $e' \notin Z$, $P_{e'}$ is a path in $T$ of length at least 2.  

Therefore each remaining element $e' \notin Z$ is added to $S$ by our procedure above as an edge linking a pair of vertices $x, y$ that are not adjacent in $T$.  
Let $e' \notin Z$.  
Since $(H,\Ss)$ and $\GB$ agree on $S$, and the endpoints of $e'$ are determined by the intersection of the fundamental circuits $C(T_e,e')$ of $M$ with respect to the collection of bases $T_e$, all of which are contained in $S$, $e'$ must have the same pair of endpoints in both $H$ and $G$.
\end{proof} 

Since $\GB$ and $\GBp$ are both determined by $S$ up to pendant roll-ups in $H[Z]$, and both are canonical representations of $N$, $\GB=\GBp$.  
This shows that $M|\cl(X)$ has a unique canonical biased graph representation.  

The same argument shows that if $G$ is a graph for $M$, and $H$ is constructed from a biased subgraph $S$ of $G$ in the same way as it is constructed above, then $S$ determines $G[\cl(X)]$ up to pendant roll-ups in the fixing tree $H[Z]$.  
Thus if $J$ is the canonical representation of $M|\cl(X)$, then applying pendant roll-ups to $J[Z]$ as necessary to place elements of long pendant lines of $M|Z$ in parallel with the frame for $M$ provided by $V(J)$, yields a biased graph $J' = G[\cl(X)]$.  
\end{proof} 

Lemma \ref{lem:unique_rep_in_span_of_fixing_set} enables us to bound the number of canonical graphs for $M$.  
Lemma \ref{lem:unique_canonical_biased_subgraph_reps} provides a more precise statement, which is more convenient to apply. 

\begin{lem} \label{lem:unique_canonical_biased_subgraph_reps}
Let $M$ be a frame matroid, and let $X$ be the fixing set of $M$.  
Let $G$ and $G'$ be canonical graphs for $M$ such that $V(G)=V(G')$ and $G[E-\cl(X)] = G'[E-\cl(X)]$. 
Then $G[\cl(X)] = G'[\cl(X)]$, and so $G = G'$. 
\end{lem} 

\begin{proof} 
By Lemma \ref{lem:unique_rep_in_span_of_fixing_set}, $M|\cl(X)$ has a unique canonical graph $J$, and both $G[\cl(X)]$ and $G'[\cl(X)]$ are obtained from $J$ by pendant roll-ups in a fixing tree $Z$ of $J$.  
(So $G$ and $G'$ have the same fixing graph, and by assumption agree outside of $\cl(X)$.) 
For each of $G$ and $G'$, modify $J$ to produce a biased graph $J'(G)$, respectively, $J'(G')$, representing $M|\cl(X)$, as follows. 
For each link $e \in E(J)$ contained in a pendant long line $l$ of $Z$, do the following.  
Let $u$ and $v$ be the endpoints of $l$.  
Denote by $E_G(x)$ the set of edges of $G$ incident to the vertex $x$. 
By assumption, for every vertex $x$, $E_G(x) - \cl(X) = E_{G'}(x)-\cl(X)$. 
Each $e \in l$ satisfies precisely one of: 
\begin{itemize} 
\item  $e \in \cl(E-E_G(u))$, 
\item  $e \in \cl(E-E_G(v))$, or 
\item  $e \notin \cl(E-E_G(u)) \cup \cl(E-E_G(v))$. 
\end{itemize}
Moreover, $e \in \cl(E-E_G(u))$ if and only if $e \in \cl(E-E_{G'}(u))$, because $G[E-\cl(X)] = G'[E-\cl(X)]$, $G$ and $G'$ share a fixing graph, and both $G$ and $G'$ are graphs for $M$. 
If $e \in \cl(E-E_G(u))$, then replace $e$ with a loop incident to $v$. 
If $e \in \cl(E-E_G(v))$, then replace $e$ with a loop incident to $u$.  
Otherwise, leave $e$ as a $u$-$v$ link. 
Now it is clear that $M$ determines the placement of $e$ in $J'(G)$, respectively, $J'(G')$, and that the resulting biased graphs $J'(G)$ and $J'(G')$ represent $M|\cl(X)$. 
Evidently, $J'(G)=J'(G')$, and by Lemma \ref{lem:unique_rep_in_span_of_fixing_set}, $G[\cl(X)] = J'(G) = J'(G') = G'[\cl(X)]$.  
\end{proof} 

Using Lemma \ref{lem:unique_canonical_biased_subgraph_reps}, we obtain the following (rather crude) upper bound on the number of canonical representations for a simple connected frame matroid.  

\begin{lem} \label{lem:boundoncanonicalrepsFRAME} 
Let $M$ be a simple connected frame matroid of rank $r$.  
The number of canonical graphs for $M$ is less than 
$\left( {r \choose 2} + r \right)^{5 {r \choose 2}} < r^{5r^2}$.  
\end{lem} 

\begin{proof} 
Every canonical graph for $M$ may be obtained as a graph $H$ on $r$ vertices, constructed as follows.  
Let $X$ be the fixing set of $M$. 
By Lemma \ref{lem:unique_rep_in_span_of_fixing_set}, $M|\cl(X)$ has a unique canonical graph.  
Let $G_1, \ldots, G_k$ be the components of the fixing graph of $M$.  
Let $H$ be the graph obtained by adding isolated vertices to the disjoint union of $G_1, \ldots, G_k$ so that the total number of vertices is $r$.  

Let $Y$ be the set of elements of $M$ not contained in the span of a component of the fixing set.  
These elements are not contained in any of the subgraphs $G_i$, and if $e$ is any edge representing an element of $Y$, and $e$ has endpoints $u, v$, then $|E(G[\{u,v\}])| \leq 5$ 
(else, since $M$ is simple, $e$ would be contained in a long line, and thus in the fixing set).  
Thus $|Y| \leq 5 {r \choose 2}$. 
Now place the elements in $Y$ as edges in $H$. 
Each may be placed as a link or as a loop, subject the constraint that induced subgraphs on each pair of vertices have at most 5 edges, so 
the number of such edge-labelled graphs $H$ 
is certainly less than $\left( {r \choose 2} + r \right)^{5 {r \choose 2}}$. 

Now place the elements in $\cl(X)$ in $H$, by replacing each edge $e_l$ of the fixing graph with its elements $e \in l \subseteq X$, and placing the elements in $\cl(X) - X$. 
By Lemmas \ref{lem:unique_rep_in_span_of_fixing_set} and  \ref{lem:unique_canonical_biased_subgraph_reps}, together the fixing graph of $M$ and the subgraph $H[E-\cl(X)]$ determine the endpoints of each edge $e \in \cl(X)$. 
Hence there are less than 
$\left( {r \choose 2} + r \right)^{5 {r \choose 2}}$
graphs for $M$. 
\end{proof} 

We can now show that excluded minors for frame matroids do not contain lines of arbitrary length.  

\begin{thm} \label{thm:exminhasnolonglineFRAME}
Let $M$ be a rank-$r$ excluded minor for the class of frame matroids.  
Then there exists a positive integer $k$, depending on $r$, such that $M$ does not contain a line of length $k$ as a restriction.  
\end{thm}

\begin{proof}
Suppose to the contrary that for some fixed $r$, for every positive integer $k$, there is an excluded minor of rank $r$ containing a line of length at least $k$ as a restriction.  
In particular then, there is an excluded minor $M$ of rank $r$ such that $M$ has a line $l$ sufficiently long for the following argument to hold.  

Since $M$ is simple, for any element $e \in l$, $M \del e$ is simple.  
By Lemma \ref{lem:boundoncanonicalrepsFRAME}, the number of canonical graphs for $M \del e$ is bounded by a function of $r$. 
Hence we may choose three elements $e, f, g \in l$ such that there are canonical biased graph representations $G_e$, $G_f$, and $G_g$, of $M \del e$, $M \del f$, and $M \del g$, respectively, such that 
$G_e \del f, g = G_f \del e, g = G_g \del e, f$, 
and such that $e, f, g$ are links in each of the biased graphs containing them.  
Let $G$ be the graph obtained by adding $e$ to $G_e$ in parallel with edges $f$ and $g$, and let $\Bb$ be the set of cycles that either do not contain $e$ and are balanced in $G_e$ or that contain $e$ and are balanced in $G_f$.  
Consider the circuits of $M$: 
\begin{itemize} 
\item  $efg$ is a circuit in both $M$ and $F\GB$, and 
\item   each circuit of $M$ containing at most two of $e$, $f$, or $g$ is a circuit of $F\GB$, since it is a circuit of one of $F(G_e)$, $F(G_f)$ or $F(G_g)$, and each of $G_e$, $G_f$ and $G_g$ agrees with $\GB$ on their respective ground sets.   
\end{itemize}
Thus we have a biased graph $(G, \Bb)$ representing $M$, a contradiction.   
\end{proof}

We can now prove Theorem \ref{thm:finitenumberexminfixedrankFRAME}.  

\begin{proof}[Proof of Theorem \ref{thm:finitenumberexminfixedrankFRAME}]
Let $r$ be a positive integer, and let $M$ be an excluded minor of rank $r$ for the class of frame matroids.  
By Theorem \ref{thm:exminhasnolonglineFRAME}, there is an integer $k$ such that $M$ does not contain a line of length $k$ as a restriction.  
Arbitrarily choose an element $e \in E(M)$, and let $\GB$ be a biased graph representing $M \del e$.  
Then $|V(G)|=r$ (we may assume $M \del e$ is non-graphic, since $M$ is not an excluded minor for the class of graphic matroids---these are all frame).  
Since $M$ has no line of length $k$, neither does $M \del e$.  
This implies $|E(G)| \leq (k-1) \cdot {r \choose 2}$.  
Hence $|E(M)| \leq (k-1) \cdot {r \choose 2} + 1$.  
There are only a finite number of matroids of rank $r$ on at most this number of elements.  
\end{proof}

It is not difficult to establish a bound on the length of a line in an excluded minor, in terms of rank, by determining how long the line $l$ in the proof of Theorem \ref{thm:exminhasnolonglineFRAME} must be to in order to guarantee the existence of the elements $e, f, g \in l$ asserted in the proof.  

\begin{cor} \label{cor:exminhasnolonglineFRAME}
Let $M$ be a rank-$r$ excluded minor for the class of frame matroids.  
Then $M$ does not contain as a restriction a line of length greater than $10 r^{5r^2}$.  
\end{cor} 

\begin{proof} 
Let $e \in E(M)$, and consider a graph $G$ for $M \del e$.  
By Lemma \ref{lem:boundoncanonicalrepsFRAME}, the number of canonical graphs for $M \del e$ is less than $r^{5r^2}$. 
Thus if $M$ has a line $l$ of length $10 r^{5r^2}$, then it is guaranteed that there that there are 11 elements $e_1, \ldots, e_{11} \in l$ such that the canonical graphs $G_{i}$ for $M \del e_i$ ($i \in \{1, \ldots, 11\}$) are identical, up to relabelling of $e_1, \ldots, e_{11}$.   

Construct an auxiliary graph $H$ on vertex set $1, \ldots, 11$, in which $i$ is adjacent to $j$ if $e_i$ is a link in $G_j$ and $e_j$ is a link in $G_i$.  
The edge set of $H$ is complementary to the edge set of the graph $H^c$ on $1, \ldots, 11$, in which two vertices $i$ and $j$ are adjacent if either $e_i$ is a loop in $G_j$ or $e_j$ is a loop in $G_i$.  
Since no graph $G_i$ has more than two of $e_1, \ldots, e_{11}$ as loops, $|E(H^c)| \leq 2|V(H)|$, and so $|E(H)| \geq {|V(H)| \choose 2} - 2|V(H)|$. 
Since this is greater than $|V(H)|^2/4$ when $|V(H)|>10$, by Tur\'an's Theorem, $H$ contains a triangle.  
Let $i, j, k$ be the vertices of this triangle.  
Then, up to relabelling of edges $e_i$, $e_j$, and $e_k$, $G_i = G_j = G_k$ and $e_i$, $e_j$, and $e_k$ are links in each of the  graphs in which they appear, as required in the proof of Theorem \ref{thm:exminhasnolonglineFRAME}.  
\end{proof} 

It is now straightforward to establish a bound, in terms of rank, on the number of elements in an excluded minor for the class of frame matroids.  
We simply substitute the bound on $k$ given by Corollary \ref{cor:exminhasnolonglineFRAME} 
into the expression bounding $|E(M)|$ in the proof of Theorem \ref{thm:finitenumberexminfixedrankFRAME}.  

\begin{cor} \label{cor:boundonsizeexminFRAME}
Let $M$ be a rank-$r$ excluded minor for the class of frame matroids.  
Then 
$|E(M)| \leq (10r^{5r^2}-1){r \choose 2} + 1 < 5r^{5r^2+2}$. 
\end{cor}

\section{Lifted-graphic matroids} \label{sec:lift} 

Throughout this section, for a lifted-graphic matroid $M$ and a graph $G$ for $M$, we always mean that $G$ is a lift graph for $M$. 
Let $M$ be a simple lifted-graphic matroid, let $G$ be a graph for $M$, let $X$ be the fixing set of $M$, and let $W$ be the fixing graph of $G$.  
Whereas a frame matroid has a unique fixing graph (Lemma \ref{lem:fixing_graphs_unique}), in contrast, a lifted-graphic matroid may have many fixing graphs: 
if $W$ is a forest, then any forest $R$ with $E(R)=E(W)$ is the fixing graph of a graph for $M|X$.  
Dealing with this fact, along with the different behaviour of loops in graphs for lifted-graphic matroids, is the business of the next few lemmas.  
For convenience, we shall leave the definition of a long line unchanged, as a line containing at least six rank-1 flats; note however, that four rank-1 flats in a long line are sufficient for the proofs in this section. 

Let $M$ be a simple lifted-graphic matroid.  
Then any graph for $M$ has at most one loop, and if $e$ is a loop in any graph for $M$, then $e$ is contained in every long line of $M$.  
Moreover, if $e$ is a loop in a graph $G$ for $M$, then the graph obtained from $G$ by removing $e$, adding a isolated vertex $v_0$, and placing $e$ as a loop incident to $v_0$ is a graph for $M$.  
Hence among all lift graphs for $M$, we choose as \emph{canonical} those that have the least number of loops, and subject to this, have exactly two components, one of which consists of either a single isolated vertex or a single vertex to which a loop is incident.  
If $G$ is a graph for $M$ that is not canonical, then a canonical graph for $M$ may be obtained from $G$ by repeatedly choosing a pair $v,v'$ of vertices, one in each of two components, and identifying $v$ and $v'$ as a single vertex, then adding a single isolated vertex $u$ and, if $G$ has a loop $e$, replacing $e$ with a loop incident to $u$.  
If $G$ is a canonical graph for $M$, then $|V(G)| = r(M) + 1$.  

We deal with the case that $M$ has just one long line separately. 
The case $M$ has more than one long line is a little bit easier. 

\begin{lem} \label{lem:liftmorethanonelongline} 
Let $M$ be a lifted-graphic matroid of rank $r$ with at least two long lines.  
The following are equivalent. 
\begin{itemize} 
\item $e$ is contained in more than one long line of $M$, 
\item $e$ is contained in all long lines of $M$, 
\item $e$ is a loop in every graph for $M$. 
\end{itemize} 
Moreover, if $M$ is simple, then $M$ has at most one element in more than one long line.  
\end{lem} 

\begin{proof} 
If $e$ is contained in more than one long line, then in any graph for $M$, $e$ must be a loop.  
This immediately implies $e$ is contained in all long lines of $M$.  
Conversely, if $e$ is a loop in a graph for $M$, then it is immediate that $e$ is contained in all long lines of $M$.  
Thus if $M$ has two elements in more than one long line, then both are loops in any graph for $M$, and so form a circuit of size 2.  
\end{proof} 

\begin{lem} \label{lem:numbercanongraphslotsalines} 
Let $M$ be a simple rank-$r$ lifted-graphic matroid containing more than one long line.  
The number of canonical graphs for $M$ is less than 
${r \choose 2}^{6{r \choose 2}} < r^{6r^2}$. 
\end{lem} 

\begin{proof} 
Every canonical graph for $M$ may be obtained as follows.  
Start with $r+1$ vertices, $v_0, v_1, \ldots, v_r$.  
Let $W$ be the collection of edges $\{e_l : l$ is long line of $M\}$. 
Place the edges in $W$ as links between distinct pairs of vertices in $\{v_1, \ldots, v_r\}$.  
Since $|W| \leq {r \choose 2}$, there are less than ${r \choose 2}^{r \choose 2}$ ways this may be done. 
For elements that are not in the fixing set of $M$, add at most five links between pairs of vertices in $\{v_1, \ldots, v_r\}$ that are not already endpoints of an edge $e_l \in W$ 
(as six such elements would be contained in the fixing set of $M$). 
Hence there are less than $5 {r \choose 2}$ elements not in the fixing set of $M$, and so less than ${r \choose 2}^{5{r \choose 2}}$ ways this may be done. 

Now if $M$ has an element $e$ contained in more than one long line (and so contained in all long lines), then by Lemma \ref{lem:liftmorethanonelongline}, $e$ must be a loop, and the only loop, in every graph for $M$: place $e$ as a loop incident to $v_0$, and replace each edge $e_l \in W$ with $|l|-1$ links between its endpoints. 
If $M$ has no element contained in more than one long line, then by Lemma \ref{lem:liftmorethanonelongline} no graph for $M$ has a loop: leave $v_0$ as an edgeless, isolated vertex, and replace each edge $e_l \in E(W)$ with $|l|$ links between its endpoints.  

Thus the number of canonical graphs for $M$ is bounded by the number of ways to place the edges of $W$ times the number of ways the place elements not in the fixing set of $M$, and this number is certainly less than ${r \choose 2}^{r \choose 2} {r \choose 2}^{5{r \choose 2}}$. 
\end{proof} 

Next we consider the case that $M$ has just one long line.  

\begin{lem} \label{lem:liftjustonelongline}
Let $M$ be a simple rank-$r$ lifted-graphic matroid containing just one long line $l$. 
Then the number of canonical graphs for $M$ is less than 
${r \choose 2}^{5{r \choose 2}} <  r^{5r^2}$. 
\end{lem} 

\begin{proof} 
Every canonical graph for $M$ may be obtained from a graph $H$ constructed as follows.  
Start with $r+1$ vertices, $v_0, v_1, \ldots, v_r$.  
Let $e_l$ be an edge for a fixing graph of $M$, corresponding to $l$.  
Place the edge $e_l$ as a link between vertices $v_1$ and $v_2$. 
Add at most five links between pairs of vertices in $\{v_1, \ldots, v_r\}$ that are not both endpoints of $e_l$.  
Next, either replace $e_l$ with $|l|$ links, or replace $e_l$ with $|l|-1$ links and place a loop incident to $v_0$.  

After placing $e_l$, there are less than $\left({r \choose 2} - 1\right)^{5{r \choose 2}}$ ways to place the elements in $E(M) - l$. 
Now consider the number of ways that the elements of $l$ may be placed.  
For each element $e \in l$, if $e \in \cl(E(H-v_1))$ or $e \in \cl(E(H-v_2))$ (equivalently, if either $H-v_1$ or $H-v_2$ contains an unbalanced cycle $C$ for which $C \cup e$ is a circuit of $M$), then $e$ must be placed as a loop if this is to be a graph for $M$. 
On the other hand, if $e$ is not in either of these closures (equivalently, neither $H-v_1$ nor $H-v_2$ contain an unbalanced cycle forming a circuit with $e$), then $e$ must be placed as a link with endpoints $v_1$ and $v_2$, since we are constructing a canonical graph for $M$. 
In other words, together $H[E(M)-l]$ and $M$ determine the placement of the elements in $l$ in $H$.  
After placing $e_l$ and the elements of $M$ not in $l$, there is just one way to place the elements of $l$.  
Hence there are less than $\left({r \choose 2} - 1\right)^{5{r \choose 2}}$ canonical graphs for $M$.  
\end{proof} 

We may now prove the key fact used in the proof of Theorem \ref{thm:finitenumberexminfixedrankLIFT}. 
For each positive integer $r$, set 
$n(r) = r^{6r^2}$.  

\begin{thm} \label{thm:exminhasnolonglineLIFT}
Let $M$ be a rank-$r$ excluded minor for the class of lifted-graphic matroids.  
Then there exists a positive integer $k$ such that $M$ does not contain a line of length $k$ as a restriction.  
\end{thm}

\begin{proof}
Suppose to the contrary that for some fixed $r$, for every positive integer $k$, there is an excluded minor of rank $r$ containing a line of length at least $k$.  
In particular then, there is an excluded minor $M$ with a line $l$ sufficiently long for the following argument to hold.  

For each $e \in l$, consider the graphs for $M \del e$.  
Since for each $e \in l$ there are, by Lemmas \ref{lem:numbercanongraphslotsalines} and \ref{lem:liftjustonelongline}, at most $n(r)$ graphs for $M \del e$, we may assume $l$ is sufficiently long that there are three elements $e, f, g \in l$, such that there are graphs $G_e$, $G_f$, and $G_g$, for $M \del e$, $M \del f$, and $M \del g$, respectively, such that 
$G_e \del f,g = G_f \del e,g = G_g \del e,f$, 
and such that $e$, $f$, and $g$ are links in each of the graphs containing them.  
Let $G$ be the graph obtained by adding $e$ to $G_e$ in parallel with edges $f$ and $g$.  
Consider the circuits of $M$: 
\begin{itemize} 
\item $efg$ is a circuit in $M$ and a theta subgraph of $G$.  
\item Each circuit of $M$ containing at most two of $e$, $f$, or $g$ is a circuit of one of $M \del e$, $M \del f$, or $M \del g$, and the graphs $G_e$, $G_f$, and $G_g$ for $M \del e$, $M \del f$, and $M \del g$, resp., agree with $G$ on their respective ground sets.  
\end{itemize}
Thus $G$ is a graph for $M$, a contradiction.   
\end{proof}

\begin{proof}
[Proof of Theorem \ref{thm:finitenumberexminfixedrankLIFT}]
Let $r$ be a positive integer, and let $M$ be an excluded minor of rank $r$ for the class of lift matroids.  
By Theorem \ref{thm:exminhasnolonglineLIFT}, there is an integer $k$ such that $M$ does not contain a line of length $k$ as a restriction.  
Arbitrarily choose an element $e \in E(M)$, and let $G$ be a canonical graph for $M \del e$.  
Then $|V(G)|=r$ (we may assume $M \del e$ is non-graphic, since we may assume $M$ is not an excluded minor for the class of graphic matroids).  
Since $M$ has no line of length $k$, neither does $M \del e$.  
This implies $|E(G)| \leq (k-1) \cdot {r \choose 2}$.  
Hence $|E(M)| \leq (k-1) \cdot {r \choose 2} + 1$.  
There are only a finite number of matroids of rank $r$ on at most this number of elements. 
\end{proof} 

Similarly to how we found the bound of Corollary \ref{cor:boundonsizeexminFRAME}, we can bound the size of an excluded minor for the class of lifted-graphic matroids.  

\begin{cor} \label{cor:exminhasnolonglineLIFT}
Let $M$ be a rank-$r$ excluded minor for the class of lifted-graphic matroids.  
Then $M$ does not contain as a restriction a line of length greater than $6 r^{6r^2}$. 
\end{cor} 

\begin{proof} 
Let $e \in E(M)$, and consider a graph $G$ for $M \del e$.  
By Lemmas \ref{lem:numbercanongraphslotsalines} and \ref{lem:liftjustonelongline}, there are less than $n(r) = r^{6r^2}$ canonical graphs for $M \del e$.  
Thus if $M$ has a line $l$ of length at least 
$6 r^{6r^2}$, 
then we are guaranteed that there are seven elements $e_1, \ldots, e_{7} \in l$ such that the canonical graphs $G_{i}$ for $M \del e_i$ ($i \in \{1, \ldots, 7\}$) are identical, up to relabelling of $e_1, \ldots, e_{7}$.   

Construct an auxiliary graph $H$ on vertex set $1, \ldots, 7$, in which $i$ is adjacent to $j$ if $e_i$ is a link in $G_j$ and $e_j$ is a link in $G_i$.  
The edge set of $H$ is complementary to the edge set of the graph $H^c$ on $1, \ldots, 7$, in which two vertices $i$ and $j$ are adjacent if either $e_i$ is a loop in $G_j$ or $e_j$ is a loop in $G_i$.  
Since no graph $G_i$ has more than one of $e_1, \ldots, e_7$ as loops, $|E(H^c)| \leq |V(H)|$, and so $|E(H)| \geq {|V(H)| \choose 2} - |V(H)|$. 
Since this is greater than $|V(H)|^2/4$ when $|V(H)|>6$, by Tur\'an's Theorem, $H$ contains a triangle.  
Let $i, j, k$ be the vertices of this triangle.  
Then, up to relabelling of $e_i, e_j, e_k$, $G_i = G_j = G_k$, and $e_i$, $e_j$, and $e_k$ are links in each of the  graphs in which they appear, as required in the proof of Theorem \ref{thm:exminhasnolonglineLIFT}.  
\end{proof} 

Substituting $k = 6 r^{6r^2}$ into the expression bounding $|E(M)|$ in the proof of Theorem \ref{thm:finitenumberexminfixedrankLIFT} gives a bound on the number of elements in an excluded minor for the class of lifted-graphic matroids: 

\begin{cor} \label{cor:boundonexminLIFT}
Let $M$ be a rank-$r$ excluded minor for the class of lifted-graphic matroids.  
Then 
$|E(M)| \leq (6r^{6r^2}-1){r \choose 2} + 1 < 3r^{6r^2+2}$. 
\end{cor}

\section{Quasi-graphic matroids} \label{sec:quasi} 

We proceed in this section analogously to Sections \ref{sec:frame} and \ref{sec:lift}. 
We bound the number of ways the elements in the fixing set may appear in a graph for a quasi-graphic matroid, then bound the number of canonical graphs for a quasi-graphic matroid in terms of its rank, for a suitable definition of \emph{canonical graph}. 
We then show that as a consequence there is a bound, in terms of rank, on the length of a line in an excluded minor. 
We will not recount here all of the properties of quasi-graphic matroids and their graphs that we require.  All basic properties and facts we use are found in \cite{Quasi}. 

Let $M$ be a simple, cosimple, and connected quasi-graphic matroid of rank $r$, and let $G$ be a graph for $M$.  
Let $c(G)$ denote the number of components of $G$. 
The number of edges in a spanning forest of $G$ is $|V(G)| - c(G)$. 
Since the edge set of a forest is independent in $M$ \cite[Lemma 2.5]{Quasi}, this number is a lower bound on the rank of $M$.  
However, the number $|V(G)|-c(G)$ tells us nothing about the number of components of $G$ that consist of just a single vertex with a single incident loop.  
Let us call a component of $G$ that consists of just a single vertex with a single incident loop an \emph{isolated loop}. 
Let $\loops(G)$ denote the set of isolated loops of $G$, let $\comp(G)$ denote the set of components of $G$ that are not isolated loops, and let $v(G) = |V(G)|-|\loops(G)|$.  

\begin{lem} \label{lem:QGverticesintermsofrank} 
Let $M$ be a simple rank-$r$ quasi-graphic matroid, and let $G$ be a graph for $M$ with no isolated vertex. 
Then 
$|\loops(G)| \leq r$, 
$|\comp(G)| \leq r$, and 
$|V(G)| \leq 2 r$.
\end{lem} 

\begin{proof} 
Since $M$ is simple and the number of components induced by a circuit is at most two, the set of isolated loops of $G$ forms an independent set. 
Therefore $|\loops(G)| \leq r$, and adding $\loops(G)$ to any forest yields an independent set.  
The number of edges in a maximal forest in $G$ is $v(G) - |\comp(G)|$, so 
$r \geq v(G) - |\comp(G)| + |\loops(G)|$.  
That is, $v(G) + |\loops(G)| \leq r + |\comp(G)|$. 
Since a set consisting of a single link from each component of $G$ that is not an isolated loop is independent, $|\comp(G)| \leq r$.  
Hence $|V(G)| = v(G) + |\loops(G)| \leq 2 r$.  
\end{proof} 

Let $M$ be a simple quasi-graphic matroid. 
Define a \emph{canonical} graph for $M$ to be a graph $G$ for $M$ satisfying: 
\begin{itemize} 
\item  $|V(G)| = 2r$; 
\item  for each long line $l$, denoting by $u_l$, $v_l$ the endpoints of $l$ in $G$, for each element $e \in l$, 
\begin{itemize} 
\item[\textbf{(CG1)}] if $e \in \cl(E(G-u_l))$ and $e \in \cl(E(G-v_l))$, then $e$ is an isolated loop; 
\item[\textbf{(CG2)}] if $e \in \cl(E(G-u_l))$ and $e \notin \cl(E(G-v_l))$, then $e$ is a loop incident to $v_l$; 
\item[\textbf{(CG3)}] if $e \notin \cl(E(G-u_l))$ and $e \in \cl(E(G-v_l))$, then $e$ is a loop incident to $u_l$; and 
\item[\textbf{(CG4)}] if $e \notin \cl(E(G-u_l))$ and $e \notin \cl(E(G-v_l))$, then $e$ is a $u_l$-$v_l$ link. 
\end{itemize}
\end{itemize} 

If $G$ is a graph for a simple quasi-graphic matroid $M$, then, as we show in the next lemma, we may obtain a canonical graph $G'$ for $M$ by adding isolated vertices to bring the total number of vertices to $2r$, and then, for each long line $l$ of $M$, replacing any edge representing an element $e \in l$ that does not respect our four conditions with one that does. 

\begin{lem} 
Let $G$ be a graph for a simple quasi-graphic matroid $M$, and let $G'$ be constructed as above from $G$.  Then $G'$ is a graph for $M$.  
\end{lem} 

\begin{proof} 
We check the four conditions $G'$ must satisfy if it is to be a framework for $M$. 

(1) We have not changed the edge set, so $E(G') = E(M)$. 

(2) Each component of $G$ that does not contain a long line is unchanged and remains a component of $G'$.  
Every component of $G$ that contains a pair of endpoints of a long line remains a component of $G'$, possibly after losing some elements that have become isolated loops in $G'$ or gaining some elements that are now loops incident to an endpoint of the long line containing them. 
Thus we have changed neither the number of vertices nor rank of the edge set of any component that contains the endpoints of a long line.  
Clearly, a component $H$ of $G'$ consisting of an isolated vertex or an isolated loop satisfies $r(E(H)) \leq |V(H)|$.  

(3) If $e \in l$ is a $u$-$v$ link in $G$, then $e$ is in neither $\cl(E(G-u))$ nor $\cl(E(G-v))$, for otherwise $G$ would violate condition (3) for frameworks. 
Hence every $u$-$v$ link of $G$ remains a $u$-$v$ link in $G'$.  
We obtain our canonical graph $G'$ from $G$ just by possibly adding isolated vertices, then rearranging the incidence of loops.  
Thus (3) certainly remains satisfied. 

(4) Suppose $C$ is a circuit of $M$ and that $G'[C]$ has more than two components.  Since $G'$ is obtained from $G$ only by rearranging loops contained in long lines, it must be the case that at least one component of $G'[C]$ is a loop $e$. 
Hence $G[C]$ is a pair of handcuffs, tight or loose, or a pair of vertex disjoint unbalanced cycles.  
Whichever the case, one of the cycles in this subgraph of $G$ is the loop $e$. 
Suppose $G[C]$ is a pair of tight handcuffs or a pair of vertex disjoint unbalanced cycles, say consisting of the two cycles $C_1$ and $e$.  
Then $G'[C]$ also consists of $C_1 \cup e$, and so has at most two components.  
So $G[C]$ must be a pair of loose handcuffs, say consisting of the cycle $C_1$, the path $P$, and $e$, where $P$ contains at least one edge. 
Since $G'[C]$ has more than two components, $C_1$ must also be a loop, let us call it $f$, whose incidence has been redefined in $G'$.  
Thus $G[C]$ consists of the pair of loops $e, f$ together with the path $P$ linking them. 
Since $e$ and $f$ were both replaced with isolated loops in $G'$, each satisfies \textbf{(CG1)} for their respective lines.  

Let $u_{l_1}$ be the endpoint of the long line $l_1$ to which $e$ is incident in $G$, and let $v_{l_1}$ be the other endpoint of $l_1$ in $G$. 
Let $u_{l_2}$ be the endpoint of the long line $l_2$ to which $f$ is incident in $G$, and let $v_{l_2}$ be the other endpoint of $l_2$ in $G$. 
(Note that we allow $l_1 = l_2$, in which case $u_{l_1}=v_{l_2}$ and $v_{l_1}=u_{l_2}$.) 
Since $e \in \cl(E(G-u_{l_1}))$, there is a cycle $C_2 \subseteq E(G-u_{l_1})$ such  that $C_2 \cup e$ is a circuit. 
If $C_2$ can be chosen such that $C_2$ either avoids $P$ or meets $P$ just at $u_{l_2}$, then do so. 
If not, choose a minimal $u_{l_1}$-$C_2$ path $Q_1$, a minimal $u_{l_2}$-$C_2$ path $Q_2$, and a subpath $Q_3$ of $C_2$ such that $Q_1 \cup Q_2 \cup Q_3$ is a $u_{l_1}$-$u_{l_2}$ path, and redefine $P$ to be this path. 
Thus $e \cup P \cup f$ is a circuit of $M$ and $P$ either avoids $C_2$ or meets $C_2$ precisely in the path $Q_3$. 

By the strong circuit elimination axiom, there is a circuit $D$ containing $f$, such that $D \subseteq (e \cup P \cup f) \cup C_2 - e$. 
The subgraph $G[D]$ is contained in $G[P \cup f \cup C_2]$, and so consists of either $C_2 \cup f$ or $C_2 \cup f \cup Q_2$. 
Now choose an edge $g \in C_2$, and apply circuit elimination to the circuits $e \cup C_2$ and $D$: there is a circuit $D_1$ contained in $e \cup C_2 \cup D - g$. 
Since the only cycles in $e \cup C_2 \cup D - g$ are $e$ and $f$, this implies $e \cup f$ is a circuit.  
But this contradicts the fact that $M$ is simple. \qedhere
\end{proof} 

We now place a bound on the number of canonical graphs for $M$.  

\begin{lem} \label{lem:boundgraphsforquasi} 
Let $M$ be a simple rank-$r$ quasi-graphic matroid.  
The number of canonical graphs for $M$ is less than
${2r \choose 2}^{{2r \choose 2}} \cdot \left({2r \choose 2} + 2r \right)^{5 {2r \choose 2}} < (3r^2)^{12r^2}$.
\end{lem} 

\begin{proof} 
By Lemma \ref{lem:QGverticesintermsofrank}, the number of vertices of a graph for $M$ is at most $2r$. 
Thus every canonical graph for $M$ may be obtained from a graph $H$ constructed as follows. 
Let $V$ be a set of $2r$ vertices. 
We first place the edges of what will be the fixing graph of a graph for $M$. 
Let $X$ be the fixing set of $M$ 
and let $W$ be the collection of edges $\{e_l : l$ is a long line in $X\}$. 
Place each edge $e_l \in W$ as a link between a pair of vertices in $V$, such that no two edges in $W$ share the same pair of endpoints. 
There can be at most ${2r \choose 2}$ edges in $W$, so 
there are less than ${2r \choose 2}^{{2r \choose 2}}$ ways to place these edges. 

For elements not in $X$, we can place, for each pair $u,v$ of vertices in $V$ that are not already endpoints of an edge $e_l \in W$, at most five edges on $\{u,v\}$, as otherwise these elements would be contained in a long line 
(where at most two of these elements may be placed as loops). 
Hence there are at most $5 {2r \choose 2}$ elements not in $X$. 
Each may be placed as a link or a loop, so the number of ways these elements may be placed is less than $\left( {2r \choose 2} + 2r \right)^{5 {2r \choose 2}}$. 

A graph $G$ for $M$ may now be obtained from $H$ as follows. 
Having placed the edges in $W$, 
we have placed the fixing graph of $G$.  
We have also placed all elements not in $X$.  
Now 
for each $e_l \in W$, we replace $e_l$ with the elements contained in $l$.  
We accomplish this as follows. 

For each $l \in X$, let $u_l, v_l$ be the endpoints of $l$ in $H$, and let $e_l^1, e_l^2, e_l^3, e_l^4$ be four ``dummy'' edges, which will temporarily stand in for, or represent, the entire line $l$. 
Since $|l| \geq 6$, and there may be at most two elements of $l$ that are loops, in every graph for $M$, every long line has at least four elements represented as links between its endpoints.  
Moreover, lengthening a line in a quasi-graphic matroid beyond four points contributes nothing new to the graphical structure of the matroid: 
A new element $e$ added to the line is in a 3-circuit with two existing edges, say $f$ and $f'$, which in a graph for $M$ has the form of either a theta subgraph (if $e$ is added as a link) or handcuffs (if $e$ is added as a loop). 
By circuit elimination, a set of edges $D \subseteq E - l$ forms a dependent set with $\{e,f\}$ if and only if $D$ forms a dependent set with $\{e,f'\}$, if and only if $D$ forms a dependent set with $\{f,f'\}$, and $e$ may be placed as either a loop or link, as appropriate, so that the circuit contained in each dependent set appears in the graph as a subgraph of one of the forms required for circuits (a balanced cycle, a theta, tight handcuffs, loose handcuffs, or a pair of vertex disjoint cycles). 
(Note that adding points to a 3-point line can potentially change the graphical structure of a quasi-graphic matroid, since a 3-point line may be represented in a graph by a balanced 3-cycle, and no graph representation for a line with more than 3 points is compatible with the graphical structure of a 3-cycle.) 

So now replace each edge $e_l$ in $H$ with four dummy edges $e_l^1, e_l^2, e_l^3, e_l^4$, placed as links between the endpoints of $e_l$. 
By the previous paragraph, there is a quasi-graphic matroid $M_H$ obtained by replacing in $M$ each long line $l$ with its 4-point dummy line. 
And, by the previous paragraph, for any set of elements $F$ that is not contained in a single long line, $F$ is dependent in $M$ if and only if, for each long line $l$ that $F$ meets, removing the elements of $l$ and replacing them with a pair of dummy elements $e_l^1, e_l^2$, results in a dependent set in $M_H$. 

Next, we replace the dummy elements of $H$ for each long line with its elements in $X$. 
We do this as follows. 
For each $l \in X$, let $u_l, v_l$ be the endpoints of $l$ in $H$. 
For each $l \in X$, and each element $e \in l$, precisely one of the following situations holds. 

Either there is a cycle $C$ in $H -\{u_l,v_l\}$ such that $C \cup e$ is independent, or there is no such cycle. 
If there is a cycle $C$ in $H -\{u_l,v_l\}$ such that $C \cup e$ is independent, then precisely one of the following holds.  
\begin{enumerate} 
\item There exists a cycle $C$ in $H-\{u_l,v_l\}$ such that $C \cup e$ is independent, and there exists a $C$-$u_l$ path $P$ in $H-v_l$ such that $C \cup e \cup P$ is a circuit: place $e$ as a loop incident to $u_l$.  (Observe that $e$ is in the closure of $E(H-v_l)$.)  
\item There exists a cycle $C$ in $H-\{u_l,v_l\}$ such that $C \cup e$ is independent, and there exists a $C$-$v_l$ path $P$ in $H-u_l$ such that $C \cup e \cup P$ is a circuit: place $e$ as a loop incident to $v_l$.  (Observe that $e$ is in the closure of $E(H-u_l)$.) 
\item For every cycle $C$ in $H - \{u_l,v_l\}$ such that $C \cup e$ is independent, for every $C$-$u_l$ path $P$, $C \cup e \cup P$ is independent, and for every $C$-$v_l$ path $Q$, $C \cup e \cup Q$ is independent. 
Now consider two subcases. 
\begin{enumerate} 
\item There is a cycle $C'$ in $H-\{u_l, v_l\}$ such that $C' \cup e$ is a circuit: place $e$ as an isolated loop. (Observe that $e$ is in the closure of $E(H-u_l)$, and in the closure of $E(H-v_l)$.)
\item For every unbalanced cycle $C$ in $H-\{u_l,v_l\}$, $C \cup e$ is independent: place $e$ as a $u_l$-$v_l$ link. 
(Observe that $e$ is in neither the closure of $E(H-u_l)$ nor the closure of $E(H-v_l)$.)
\end{enumerate} 
\end{enumerate}
It is not hard to see that (1) and (2) cannot both hold, by applying the circuit elimination axiom to appropriate subgraphs of $H$
(thus in each case, $e$ is in the closure of just one endpoint of $l$). 

If there is no cycle $C$ in $H-\{u_l, v_l\}$ such that $C \cup e$ is independent, then precisely one of the following holds. 
\begin{enumerate} \setcounter{enumi}{3}
\item All cycles in $H-\{u_l,v_l\}$ are circuits. 
\item There is an unbalanced cycle in $H-\{u_l,v_l\}$, and for every unbalanced cycle $C$ in $H-\{u_l,v_l\}$, $C \cup e$ is a circuit 
\end{enumerate}
If (4) holds, then precisely one of the following holds. 
For each long line $l$, let $D_l = \{e_l^1, e_l^2, e_l^3, e_l^4\}$. 
As for cases (1), (2), and (3) (a) and (b) above, in each of the following cases, it is evident that the placement of $e$ satisfies the conditions \textbf{(CG1)}-\textbf{(CG4)}. 
\begin{enumerate}[(a)]
\item $H \del D_l$ is balanced. 
There is no cycle in $H \del D_l$ forming a circuit with $e$, nor any cycle $C$ together with a minimal $C$-$\{u_l,v_l\}$ path forming a circuit with $e$: place $e$ as a $u_l$-$v_l$ link. 
\item Every unbalanced cycle in $H \del D_l$ meets both $u_l$ and $v_l$: place $e$ as a $u_l$-$v_l$ link. (In this case, for every unbalanced cycle $C$ in $H \del D_l$, $C \cup e$ is dependent; placing $e$ as a $u_l$-$v_l$ link means $C \cup e$ either contains a balanced cycle or is a theta subgraph with all there cycle unbalanced.) 
\item Every unbalanced cycle in $H \del D_l$ meets $u_l$ but avoids $v_l$. If there is an unbalanced cycle $C$ with $C \cup e$ a circuit, place $e$ as a loop incident to $u_l$; otherwise place $e$ as a $u_l$-$v_l$ link. 
\item Every unbalanced cycle in $H \del D_l$ meets $v_l$ but avoids $u_l$. If there is an unbalanced cycle $C$ with $C \cup e$ a circuit, place $e$ as a loop incident to $v_l$; otherwise place $e$ as a $u_l$-$v_l$ link. 
\item Each of $H-u_l$ and $H-v_l$ contain an unbalanced cycle, say $D_u$ and $D_v$, respectively. 
\begin{enumerate}[(i)]
\item If both $D_u \cup e$ and $D_v \cup e$ are independent: place $e$ as a $u_l$-$v_l$ link. 
\item If $D_u \cup e$ is a circuit while $D_v \cup e$ is independent: place $e$ as a loop incident to $u_l$. 
\item If $D_u \cup e$ is independent while $D_v \cup e$ is a circuit: place $e$ as a loop incident to $v_l$. 
\item If both $D_u \cup e$ and $D_v \cup e$ are circuits: place $e$ as an isolated loop. 
\end{enumerate} 
\end{enumerate}
Finally, if (5) holds, then $e$ must a loop. 
This further implies that for any unbalanced cycle $D$ meeting $u_l$ or $v_l$, $D \cup e$ is a circuit, so $e$ could be a loop incident to $u_l$, $v_l$ or any other vertex. 
Place $e$ as an isolated loop. 

(Geometrically, the endpoints $u_l$ and $v_l$ of the long line $l$ may be thought of as two points at which the span of the points in $l$ meet the span of the elements of $M$ that are not in $l$. The points of $l$ that are $u_l$-$v_l$-links are the points of $M$ that are minimally in the span of $\{u_l, v_l\}$. Deleting an endpoint $u_l$ corresponds to removing all points of $l$ aside from possibly a point parallel to $v_l$. Asking if an element $e \in l$ is in the closure of the remaining elements, is asking whether $e$ is in the span of $(E(M) - l) \cup v_l$. The placement of $e$ according to (1)-(4), then, just places $e$ in the graph appropriately---according to the form circuits take in a graph---so that the graph so constructed is a graph for $M$.) 

By construction, the resulting graph $G$ is a canonical graph for $M$.  
Moreover, $G$ is determined by the placement of the edges in $W$ together with the subgraph $H[E-X]$. 
Thus the number of canonical graphs for $M$ is certainly bounded by the number of such graphs $H$ that may be constructed.  
This number is at most the number of ways to place the edges in $W$, times the number of ways to place the elements not in $X$.  
Using the upper bounds on these numbers established above, we see that this product is less than 
${2r \choose 2}^{{2r \choose 2}} \cdot \left({2r \choose 2} + 2r \right)^{5 {2r \choose 2}}$. 
\end{proof} 

We may now bound the length of line that may appear in an excluded minor. 
Put $n(r) = {2r \choose 2}^{{2r \choose 2}} \cdot \left({2r \choose 2} + 2r \right)^{5 {2r \choose 2}}$. 

\begin{thm} \label{thm:exminhasnolonglineQUASI}
Let $M$ be a rank-$r$ excluded minor for the class of quasi-graphic matroids.  
Then there is a positive integer $k$ such that $M$ does not contain a line of length $k$ as a restriction. 
\end{thm} 

\begin{proof} 
Suppose to the contrary that for some fixed $r$, for every positive integer $k$, there is an excluded minor of rank $r$ containing a line of length at least $k$.  
In particular then, there is an excluded minor $M$ with a line $l$ sufficiently long for the following argument to hold.  

For each $e \in l$, consider the canonical graphs for $M \del e$.  
By Lemma \ref{lem:ex_min_simple_cosimple_QUASI}, $M$ is simple, so $M \del e$ is simple. 
Since for each $e \in l$ there are, by Lemma \ref{lem:boundgraphsforquasi}, at most $n(r)$ canonical graphs for $M \del e$, we may assume $l$ is sufficiently long that there are three elements $e, f, g \in l$, such that there are graphs $G_e$, $G_f$, and $G_g$, for $M \del e$, $M \del f$, and $M \del g$, respectively, such that 
$G_e \del f,g = G_f \del e,g = G_g \del e,f$, and such that $e$, $f$, and $g$ are links in each of the graphs containing them.  
Let $G$ be the graph obtained by adding $e$ to $G_e$ in parallel with edges $f$ and $g$.  
Then $G$ is a graph for $M$, a contradiction.   
\end{proof} 

\begin{proof}[Proof of Theorem \ref{thm:finitenumberexminfixedrankQUASI}] 
Let $r$ be a positive integer, and let $M$ be an excluded minor of rank $r$ for the class of quasi-graphic matroids.  
By Theorem \ref{thm:exminhasnolonglineQUASI}, there is an integer $k$ such that $M$ does not contain a line of length $k$ as a restriction.  
Arbitrarily choose an element $e \in E(M)$, and let $G$ be a graph for $M \del e$.  
By Lemma \ref{lem:QGverticesintermsofrank}, $|V(G)| \leq 2r$. 
Since $M$ has no line of length $k$, neither does $M \del e$.  
This implies $|E(G)| \leq (k-1) \cdot {2r \choose 2}$.  
Hence $|E(M)| \leq (k-1) \cdot {2r \choose 2} + 1$.  
There are only a finite number of matroids of rank $r$ on at most this number of elements. 
\end{proof} 

As for excluded minors for the classes of frame and lifted-graphic matroids, it is now not difficult to establish a bound on the size, in terms of rank, of an excluded minor for the class of quasi-graphic matroids.  

\begin{cor} \label{cor:exminhasnolonglineQUASI}
Let $M$ be a rank-$r$ excluded minor for the class of quasi-graphic matroids.  
Then $M$ does not contain as a restriction a line of length greater than 
$(8r+2) \cdot n(r) < (3r^2)^{12r^2+1}$. 
\end{cor} 

\begin{proof} 
Let $e \in E(M)$, and consider a graph $G$ for $M \del e$.  
By Lemma \ref{lem:boundgraphsforquasi}, there are less than 
$n(r)$ 
graphs for $M \del e$. 
Hence if $M$ has a line $l$ of length at least 
$(8r+2) \cdot n(r)$, then it is guaranteed that there are $8r+3$ elements $e_1, \ldots, e_{8r+3} \in l$ such that the graphs $G_{i}$ for $M \del e_i$ ($i \in \{1, \ldots, 8r+3\}$) are identical, up to relabelling of $e_1, \ldots, e_{8r+3}$.   

Construct an auxiliary graph $H$ on vertex set $1, \ldots, 8r+3$, in which $i$ is adjacent to $j$ if $e_i$ is a link in $G_j$ and $e_j$ is a link in $G_i$.  
The edge set of $H$ is complementary to the edge set of the graph $H^c$ on $1, \ldots, 8r+3$, in which two vertices $i$ and $j$ are adjacent if either $e_i$ is a loop in $G_j$ or $e_j$ is a loop in $G_i$.  
Each graph $G_i$ has at most $2r$ vertices, and $M \del e$ is simple, so each graph $G_i$ has at most $2r$ loops.  
Since no graph $G_i$ has more than $2r$ of $e_1, \ldots, e_{8r+3}$ as loops, $|E(H^c)| \leq 2r|V(H)|$, and so $|E(H)| \geq {|V(H)| \choose 2} - 2r|V(H)|$. 
Since this is greater than $|V(H)|^2/4$ when $|V(H)|>8r+2$, by Tur\'an's Theorem, $H$ contains a triangle.  
Let $i, j, k$ be the vertices of this triangle.  
Then $G_i = G_j = G_k$ and $e_i$, $e_j$, and $e_k$ are links in each of the  graphs in which they appear, as required in the proof of Theorem \ref{thm:exminhasnolonglineQUASI}.  
\end{proof} 

Substituting 
$k = (8r+2) \cdot n(r)$ 
into the expression bounding $|E(M)|$ in the proof of Theorem \ref{thm:finitenumberexminfixedrankQUASI} gives a bound on the number of elements in an excluded minor for the class of quasi-graphic matroids: 

\begin{cor} \label{cor:boundonexminQUASI}
Let $M$ be a rank-$r$ excluded minor for the class of quasi-graphic matroids.  
Then 
$|E(M)| < (8r+2) n(r) {2r \choose 2} < (3r^2)^{12r^2+2}$. 
\end{cor} 

\bibliographystyle{amsplain}
\bibliography{everything}

\end{document}